\documentclass[11pt]{amsart}
\usepackage{graphicx}
\usepackage{amsmath,amsthm}
\numberwithin{equation}{section}

\newtheorem{theorem}{Theorem}[section]

\newtheorem{lemma}[theorem]{Lemma}
\newtheorem{proposition}[theorem]{Proposition}

\theoremstyle{definition}

\newtheorem{example}[theorem]{Example}

\theoremstyle{remark}
\newtheorem{remark}[theorem]{Remark}

\def\to{{\rightarrow}}
\def\R{\mathbb{R}}
\def\Z{\mathbb{Z}}

\def\N{\mathbb{N}}
\def\Q{\mathbf{Q}}
\def\S{\mathbf{S}}

\def\F{\mathbf{F}}
\def\Id{\mathbf{I}}
\def\P{\mathbf{P}}
\def\U{\mathbf{U}}
\def\A{\mathbf{A}}
\def\X{\mathbf{X}}
\def\M{\mathbf{M}}

\def\a{\mathbf{a}}
\def\b{\mathbf{b}}
\def\x{\mathbf{x}}
\def\y{\mathbf{y}}
\def\z{\mathbf{z}}
\def\c{\mathbf{c}}
\def\e{\mathbf{e}}
\def\f{\mathbf{f}}

\title{A MAX-CUT formulation of 0/1 programs}

\author{Jean B. Lasserre}
\address{LAAS-CNRS and Institute of Mathematics,
LAAS, 7 Avenue du Colonel Roche, BP 54200,
31031 Toulouse C\'{e}dex 4, France.}
\email{lasserre@laas.fr}

\date{\today}

\begin{document}

\begin{abstract}
We show that the linear or quadratic 0/1 program
\[\P:\quad\min\{ \c^T\x+\x^T\F\x : \:\A\,\x =\b;\:\x\in\{0,1\}^n\},\]
can be formulated as a MAX-CUT problem
whose associated graph
is simply related to the matrices $\F$ and $\A^T\A$.
Hence the whole arsenal 
of approximation techniques for MAX-CUT can be applied.
We also  compare the lower bound
of the resulting semidefinite (or Shor) relaxation with that of the standard LP-relaxation and the first semidefinite relaxations
associated with the Lasserre hierarchy and the copositive formulations of $\P$.
\end{abstract}
\date{}
\maketitle

\section{Introduction}

Consider the linear or quadratic 0/1 program $\P$ defined by:
\begin{equation}
\label{def-pb}
\P:\qquad f^*\,=\,\min_\x\,\{\,\c^T\x+\x^T\F\x:\:\A\,\x\,=\,\b;\quad\x\in\{0,1\}^n\,\}
\end{equation}
for some cost vector $\c\in\R^n$, some matrix $\A\in\Z^{m\times n}$, some vector $\b\in\Z^m$, and some
real symmetric matrix $\F\in\R^{n\times n}$. If $\F=0$ then $\P$ is a 0/1 linear program
and a quadratic 0/1 program otherwise. 
Obtaining good quality lower bounds on $f^*$ is highly desirable since 
the efficiency of Branch \& Bound algorithms to solve large scale problems $\P$ heavily depends on the quality 
of bounds of this form computed at nodes of the search tree. 

To obtain lower bounds for 0/1 programs (\ref{def-pb}) one may solve a relaxation of $\P$
where the integrality constraints $\x\in\{0,1\}^n$ are replaced with the box constraints
$\x\in [0,1]^n$. If $\F=0$ the resulting relaxation is linear whereas if 
$\F$ is positive definite it is a (convex) quadratic program. If
$\F$ is not positive semidefinite then one may also solve a convex quadratic program but now 
 with an appropriate convex quadratic underestimator $\x^T\tilde{\F}\x$ of $\x^T\F\x$ on $[0,1]^n$.
 An alternative is to consider an equivalent formulation of $\P$ as a copositive conic program as advocated by Burer \cite{burer}
 and compute a sequence of lower bounds by solving an appropriate hierarchy of LP- or SDP-relaxations associated 
 with the copositive cone (or its dual). For more details on the latter approach the interested reader is referred to e.g. 
 De Klerk and Pasechnik \cite{de-pasech}, D\"ur \cite{durr}, Bomze \cite{bomze1}, and Bomze and de Klerk \cite{bomze}.
 
{\bf Contribution.} The purpose of this note is to show that solving $\P$ is equivalent to minimizing a quadratic form 
in $n+1$ variables on the hypercube $\{-1,1\}^{n+1}$ (and the quadratic form is explicit from the data of $\P$). In other words $\P$ can be viewed as an explicit instance of the MAX-CUT problem. This idea had been already briefly mentioned in Rendl \cite{rendl2}
in the context of partitioning and ordering problems.
Hence the MAX-CUT problem which at first glance seems to be 
a very specific combinatorial optimization problem, in fact can be considered as a canonical model of linear and quadratic 0/1 programs. 
In particular, to each linear or quadratic 0/1 program (\ref{def-pb}) one may associate a graph $G=(V,E)$ with $n+1$ nodes and $(i,j)\in E$
whenever a product $x_ix_j$ has a nonzero coefficient in some quadratic form built upon the data
$\c,\b,\F$ and $\A$ of (\ref{def-pb}). (Among other things, the sparsity of $G$ is related to the sparsity of the matrices
$\F$ and $\A^T\A$.) Then solving (\ref{def-pb}) reduces to finding a maximum (weighted) cut of $G$. 

Therefore the whole specialized arsenal of approximation techniques for MAX-CUT 
can be applied. In particular one obtains a lower bound $f^*_1$ on $f^*$ by solving the standard (Shor) SDP-relaxation
associated with the resulting MAX-CUT problem while solving higher levels of the associated Lasserre-SOS hierarchy 
\cite{lasserre1,lasserre2} would provide
a monotone nondecreasing sequence of improved lower bounds $f^*_1\leq f^*_d\leq f^*$, $d=2,\ldots$, but of course at a higher computational cost. 
Alternatively one may also apply the Handelman hierarchy of LP-relaxations as described and analyzed in Laurent and Sun \cite{laurent}.
For more details on recent developments on computational approaches to MAX-CUT the interested reader is referred to 
Palagi et al. \cite{rendl}. In addition one may also obtain
performance guarantees {\it \`a la Nesterov} \cite{nesterov} 
or their improvements by Marshall \cite{marshall}. 
Finally, the same methodology also works for general 0/1 optimization problems with feasible set as in (\ref{def-pb}) and polynomial criterion $f\in\R[\x]$ of degree $d>2$, except that now the problem reduces to minimizing a new polynomial criterion
$\tilde{f}(\x)$ on the hypercube $\{-1,1\}^n$ (and not a MAX-CUT problem any more as the degree of $\tilde{f}$ is larger than $2$).

{\bf Numerical experiments.} For 0/1 linear programs ($\F=0$) the lower bound $f^*_1$ can be better than the 
standard LP-relaxation which consists in replacing the integrality constraints $\x\in\{0,1\}^n$ with the box $[0,1]^n$, as shown
on a (limited) sample of 0/1-knapsack-type examples. 
In fact $f^*_1$ is almost always better than the lower bound obtained from
the first SDP-relaxation of the Lasserre-SOS hierarchy applied to the initial formulation (\ref{def-pb}) of the problem (an SDP
of same size which is a basic quadratic relaxation of the initial problem). This is good news since typically the SOS-hierarchy is known to produce good lower bounds
for general polynomial optimization problems (discrete or not) even at the first level of the hierarchy. 
Even more, 
the first level SDP-relaxation has the celebrated Goemans \& Williamson performance guarantee ($\approx 87\%$)
when the matrix $\Q$ (associated with the quadratic form) has nonnegative entries and 
a performance guarantee $\approx 64\%$ when $\Q\succeq0$. (However note that the matrix $\Q$
associated with our MAX-CUT problem equivalent to the initial 0/1 program (\ref{def-pb}) does not have 
all its entries nonnegative.) 
For quadratic 0/1 knapsacks, $f^*_1$ is also better than the lower bound obtained
by relaxing $\{0,1\}^n$ to $[0,1]^n$, replacing $\F$ with a convex quadratic underestimator, and solving the resulting convex quadratic relaxation.

We have also considered the MAX-CUT formulation of the $k$-cluster problem
$\rho=\min\{\x^T\A\x: \e^T\x=k;\x\in\{0,1\}^n\}$ where $k\in\N$ and $\A$ is some real symmetric matrix,
for $n=70$ and $n=100$ variables and with various percentages of zero entries in the matrix $\A$. In all examples 
the optimal value of the Shor relaxation was almost indistinguishable of
the first semidefinite relaxation of the Lasserre-SOS hierarchy applied to the initial formulation.
We have also compared the Shor relaxation with a quadratic  convex relaxation of the problem.
The former always provides better lower bounds (and sometimes significantly better). At last we have also compared the Shor relaxations respectively associated with the MAX-CUT 
and copositive formulations. Again the respective optimal values are very close (less than $1\%$ relative difference
and in a few cases the former is significantly better).

\section{Main result}

Denote by $\Z$ the set of integer numbers and $\N\subset\Z$ the set of natural numbers. Let $\P$ be the 0/1 program defined in (\ref{def-pb}) with $\F^T=\F\in\R^{n\times n}$,
$\A\in\Z^{m\times n}$, $\c\in\R^n$ and $\b\in\Z^m$. Let $\vert\c\vert:=(\vert c_i\vert)\in\R^n_+$.
With $\e\in\Z^n$ being the vector of all ones, notice first that $\P$ has an equivalent formulation 
on the hypercube $\{-1,1\}^n$, by the change of variables $\tilde{\x}:=2\x-\e$.
Indeed, $\A$, $\b$, $\c$ and $\F$ 
now become $\A/2$, $\b-\A\e/2$, $(\c+\e^T\F)/2$ and  $\F/4$ respectively.

Therefore from now on we consider the discrete program:

\begin{equation}
\label{def-newpb}
\P:\qquad f^*=\min_{\x\in\{-1,1\}^n}\:\{\: \c^T\x+\x^T\F\x:\: \A\,\x\,=\,\b\},
\end{equation}
on the hypercube $\{-1,1\}^n$, with $\A\in\Z^{m\times n}$, $\c\in\R^n$, $\b\in\Z^m$, and $\F^T=\F\in\R^{n\times n}$.
With $\c$ and $\F$, let us associate the scalars:
\begin{eqnarray*}
r^1_{\c,\F}&=&\min\,\{\,\c^T\x+\langle \X,\F\rangle: \left[\begin{array}{cc}1&\x^T\\
\x&\X\end{array}\right]\,\succeq0;\:X_{ii}=1,\:i=1,\ldots,n\}\\
r^2_{\c,\F}&=&\max\,\{\,\c^T\x+\langle \X,\F\rangle: \left[\begin{array}{cc}1&\x^T\\
\x&\X\end{array}\right]\,\succeq0;\:X_{ii}=1,\:i=1,\ldots,n\}
\end{eqnarray*}
(with $\X^T=\X$) and let
\begin{equation}
\label{def-rcf}
\rho(\c,\F)\,:=\,\max_{i=1,2}\vert r^i_{\c,\F}\vert.
\end{equation}
It is straightforward to verify that 
\begin{equation}
\label{bound-rho}
\rho(\c,\F)\geq\max\,\{\,\vert\c^T\x+\x^T\F\x\,\vert:\:\x\in\{-1,1\}^n\,\},\end{equation}
and $\rho(\c,\F)=\vert\c\vert$ if $\F=0$. Moreover each scalar $r^i_{\c,\F}$ 
can be computed by solving an SDP which is the Shor relaxation (or first level of the Lasserre-SOS hierarchy \cite{lasserre1,lasserre2})
associated with the problems $\min\,(\max)\{\c^T\x+\x^T\F\x:\x\in\{-1,1\}^n\}$.

\subsection{A MAX-CUT formulation of $\P$}
\begin{lemma}
\label{lemma-1}
Let $\P$ be as (\ref{def-newpb}) and
let $\rho(\c,\F)$ be as in (\ref{def-rcf}). If $f^*<+\infty$
then $f^*$ is the optimal value of the 
quadratic minimization problem, i.e.,
\begin{equation}
\label{quad-pb}
f^*\,=\,\theta:=\min_{\x\in \{-1,1\}^n}\:\c^T\x+\x^T\F\x+(2\,\rho(\c,\F)+1)\cdot\Vert \A\,\x-\b\Vert^2.
\end{equation}
Moreover $\P$ has no feasible solution ($f^*=+\infty$)  if and only if $\theta>\rho(\c,\F)$.
\end{lemma}
\begin{proof}
Let $\Delta:=\{\x\in\{-1,1\}^n: \A\x=\b\}$ be the feasible set of $\P$ defined in (\ref{def-newpb}), and
let $f:\R^n\to\R$ be the function 
\begin{equation}
\label{def-f}
x\mapsto f(\x):=\c^T\x+\x^T\F\x+(2\,\rho(\c,\F)+1)\cdot\Vert\A\x-\b\Vert^2.\end{equation}
On $\{-1,1\}^n$ one has $\max\{\c^T\x+\x^T\F\x:\x\in\{-1,1\}^n\}\leq \rho(\c,\F)$, and 
\[\Vert\A\x-\b\Vert^2\geq 1,\quad \forall \,\x\in\,\{-1,1\}^n\setminus\Delta,\]
because $\A\in \Z^{m\times n}$ and $\b\in\Z^m$.
Therefore,
\begin{equation}
\label{value}
f(\x)\:\left\{\begin{array}{l}
=\c^T\x+\x^T\F\x \mbox{ on }\Delta\\
\geq \c^T\x+\x^T\F\x+2\,\rho(\c,\F)+1> \rho(\c,\F)\mbox{ on }\{-1,1\}^n\setminus \Delta.\end{array}\right.\end{equation}
From this and $\c^T\x +\x^T\F\x\leq\rho(\c,\F)$ on $\Delta$, the result follows. 
\end{proof}
Next, let $Q:\R^{n+1}\to\,\R$ be the homogenization of the quadratic polynomial $f$, i.e.,
the quadratic form $Q(\x,x_0):=x_0^2f(\frac{\x}{x_0})$, or in explicitly form:
\begin{equation}
\label{quadraticform}
Q(\x,x_0)\,=\,x_0\,\c^T\x+\x^T\F\x+(2\,\rho(\c,\F)+1)\cdot \Vert\A\x-x_0\,\b\Vert^2.
\end{equation}
Observe that $Q(\x,1)=f(\x)$.
\begin{theorem}
\label{th-linear}
Let $f^*=\min\{\c^T\x+\x^T\F\x:\,\A\x=\b;\x\in\{-1,1\}^n\}$ and let $Q$ be the quadratic form in (\ref{quadraticform}). If $f^*<+\infty$ then
\begin{equation}
\label{obs}
f^*\,=\,\theta:=\min_{(\x,x_0)\in\{-1,1\}^{n+1}}\,Q(\x,x_0),
\end{equation}
that is, $f^*$ is the optimal value of the MAX-CUT problem associated with the quadratic form $Q$. 
Moreover $f^*=+\infty$ if and only if $\theta>\rho(\c,\F)$.
\end{theorem}
\begin{proof}
Let $f$ be as in (\ref{def-f}). By definition of $Q$,
\begin{equation}
\label{aux1}
\min_{\x\in\{-1,1\}^{n}}\,f(\x)\,=\,\min_{(\x,x_0)\in\{-1,1\}^{n+1}}\,\{\,Q(\x,x_0):\:x_0=1\,\}.\end{equation}
On the other hand, let $(\x^*,x_0^*)\in\{-1,1\}^{n+1}$ be a global minimizer of 
$\min\{Q(\x,x_0):(\x,x_0)\in\{-1,1\}^{n+1}\}$.
Then by homogeneity of $Q$, $(-\x^*,-x_0)$ is also a global minimizer and so 
one may decide arbitrarily to fix $x_0=1$. That is,
\[\min_{(\x,x_0)\in\{-1,1\}^{n+1}}\,Q(\x,x_0)\,=\,\min_{(\x,x_0)\in\{-1,1\}^{n+1}}\,\{\,Q(\x,x_0):\:x_0=1\,\},\]
which combined with (\ref{aux1}) and Lemma \ref{lemma-1} yields the desired result.
\end{proof}
Next, write $Q(\x,x_0)=(\x,x_0)\Q(\x,x_0)^T$
for an appropriate real symmetric matrix $\Q\in\R^{(n+1)\times (n+1)}$,
and introduce the semidefinite programs
\begin{equation}
\label{lower}
\min_{\X}\:\{\langle \Q,\X\rangle \:: \X\succeq0;\:
X_{ii}\,=\,1,\quad i=1,\ldots,n+1\}
\end{equation}
with optimal value denoted by $\min\Q_+$,
and
\begin{equation}
\label{upper}
\max_{\X}\:\{\langle \Q,\X\rangle \:: \X\succeq0;\:
X_{ii}\,=\,1,\quad i=1,\ldots,n+1\}
\end{equation}
with optimal value denoted by $\max\Q^+$.
\begin{proposition}
\label{prop-main}
Let $\P$ be the problem defined in (\ref{def-newpb}) with optimal value $f^*$ and let $\rho(\c,\F)$ be as in (\ref{def-rcf}).
If $f^*<+\infty$ then
\begin{equation}
\label{thmain-1}
\min\Q_+\leq f^*\leq \frac{2}{\pi}\min\Q_++(1-\frac{2}{\pi})\max\Q^+,
\end{equation}
where $\Q$ is the real symmetric matrix asociated with the quadratic form (\ref{quadraticform}) and $\min\Q_+$ (resp. $\max\Q^+$)  is the optimal value of the semidefinite  program (\ref{lower}) (resp. (\ref{upper})).
Moreover, if $\min\Q_+ >\rho(\c,\F)$ then $\P$ has no feasible solution (i.e., $f^*=+\infty$).
\end{proposition}
\begin{proof}
The bounds in (\ref{thmain-1}) are from Nesterov \cite{nesterov}. In addition, one may also use the bounds provided in Marshall \cite{marshall} which sometimes improve those in (\ref{thmain-1}).
Next, let $\Delta:=\{\x\in\{-1,1\}^n: \A\x=\b\}$ and assume that $\Delta\neq\emptyset$. Then by
(\ref{bound-rho}) and (\ref{value}),  $\min\Q_+\leq f^*\leq\rho(\c,\F)$.
Therefore $\Delta=\emptyset$ if $\min\Q_+>\rho(\c,\F)$.
\end{proof}
The quality of the upper bound in (\ref{thmain-1}) depends strongly on the magnitude of the ``penalty coefficient" $2\rho(\c,\F)+1$ 
in the definition of the function $f$ in (\ref{def-f}). Whether or not the penalty parameter $\rho(\c,\F)$ could be 
taken smaller (at least in some cases) has not been investigated and is beyond the scope of this paper. 
However for a practical use of relaxations 
what matters most is the quality of the {\it lower} bound $\min\Q_+$ which in principle is very good for MAX-CUT problems
(even if $\Q\not\geq0$ or $\Q\not\succeq0$).
For instance in a Branch \& Bound algorithm
the lower bound $\min\Q_+$ has an important impact in the pruning of nodes in the search tree.

\subsection{Sparsity}
Hence to each 0/1 program (\ref{def-pb}) one may associate a graph $G=(V,E)$ with $n+1$ nodes and an arc $(i,j)\in E$
connects the nodes $i,j\in V$ if and only if the coefficient $\Q_{ij}$ of the quadratic form $Q(\x,x_0)$ does not vanish.
{\it Sparsity} properties of $G$ are of primary interest, e.g. for computational reasons.
From the definition of the matrix $\Q$, this sparsity is in turn related to sparsity 
of the matrix $\F+(2\rho(\c,\F)+1)\cdot \A^T\A$, hence of sparsity of $\F$ and $\A^T\A$. 

In particular, two nodes $i,j$ are not connected if $\F_{ij}=0$ and 
$\A_{ki}\A_{kj}=0$ for all $k=1,\ldots,m$, that is, in any of the constraints $\A\x=\b$, 
at most one of the two variables $(x_i,x_j)$ appears (a structural condition). A weaker condition 
(non-structural as it depends on
values of entries of $\A$) is that $\sum_k\A_{ki}\A_{kj}=0$.

\subsection{Extension to inequalities}

Let $f(\x):=\c^T\x+\x^T\F\x$ for some $\c\in\R^n$ and some
$\F^T=\F\in\R^{n\times n}$, and consider the problem:
\begin{equation}
\label{def-pb2}
\P:\qquad f^*\,=\,\min_\x\,\{\,f(\x):\:\A\,\x\,\leq\,\b;\quad\x\in\{0,1\}^n\,\},
\end{equation}
for some cost vector $\c\in\Z^n$, some matrix $\A\in\Z^{m\times n}$, and some vector $\b\in\Z^m$.
We may and will replace (\ref{def-pb2}) with the equivalent pure  integer program:
\[\P':\qquad f^*\,=\,\min_{\x,\y}\,\{\,f(\x):\:\A\,\x+\y\,=\,\b;\quad\x\in\{0,1\}^n;\:\y\in\N^m\,\}.\]
Next, as $\x\in\{0,1\}^n$ we can bound each integer variable $y_j$ by $M_j:=b_j-\min \{\A_j\,\x:\x\in\{0,1\}^n\}$, $j=1,\ldots,m$,
where $\A_j$ denotes the $j$-th row vector of the matrix $\A$; and in fact
$M_j=b_j-\sum_i\min[0,\A_{ji}]$, $j=1,\ldots,m$. Then we may use the standard decomposition of $y_j$ into a weighted sum of boolean variables :
\[y_j=\sum_{k=0}^{s_j}2^k z_{jk},\qquad z_{jk}\in\{0,1\}^n,\quad j=1,\ldots,m,\]
(where $s_j:=\lceil\log(M_j)\rceil$) and replace (\ref{def-pb2}) with the equivalent 0/1 program:
 \[f^*\,=\,\min_{\x,\z}\,\{\,f(\x):\:\A_j^T\x+\sum_{k=0}^{s_j}2^k z_{jk}\,=\,b_j, \:j\leq m;\:(\x,\z)\in\{0,1\}^{n+s}\,\},\]
 (where $s:=\sum_j(1+s_j)$) which is of the form (\ref{def-pb}).
 
 \subsection{Extension to polynomial programs}
 
 Let $f\in\R[\x]$ be a polynomial of even degree $d>2$, written as
 \[\x\mapsto f(\x)\,=\,\sum_{\alpha\in\N^n}f_\alpha\,\x^\alpha,\]
 for some vector of coefficients $\f=(f_\alpha)\in\R^{s(d)}$ (with $s(d)={n+d\choose n}$). Given a sequence
 $\y=(y_\alpha)$, $\alpha\in\N^n$, define the Riesz functional
 $L_\y:\R[\x]\to\R$ by:
 \[f\mapsto L_\y(f)\,=\,\sum_{\alpha\in\N^n} f_\alpha\,y_\alpha,\qquad f\in\R[\x].\]
 Consider the polynomial program:
 \begin{equation}
 \label{def-pb-f}
f^*\,=\,\min\:\{\,f(\x):\: \A\x=\b;\quad\x\in\{-1,1\}^n\,\},
\end{equation}
on the hyper cube $\{-1,1\}^n$. Let $d':=d/2$, $\x\mapsto g_j(\x):=1-x_j^2$, $j=1,\ldots,n$, 
and with $f$  let us associate the scalars:
\begin{eqnarray*}
r^1_{f}&=&\min\,\{\,L_\y(f): \M_{d'}(\y)\succeq0;\:\M_{d'-1}(g_j\,\y)=0,\:j=1,\ldots,n\,\}\\
r^1_{f}&=&\max\,\{\,L_\y(f): \M_{d'}(\y)\succeq0;\:\M_{d'-1}(g_j\,\y)=0,\:j=1,\ldots,n\,\}
\end{eqnarray*}
where $\M_{d'}(\y)$ (resp. $\M_{d'-1}(g_j\,\y)$) is the moment matrix (resp. localizing matrix) of order $d'$ associated with 
the real sequence $\y=(y_\alpha)$, $\alpha\in\N^n$, (resp. with the sequence $\y$ and the polynomial $g_j$). 
It turns out that $r^1_f$ (resp. $r^2_f$) is the optimal value of the first SDP-relaxation of the Lasserre-SOS hierarchy
associated with the optimization problem $\min\,({\rm resp.}\,\max)\{f(\x):\x\in\{-1,1\}^n\}$ and so
$r^1_f\leq\min\{f(\x):\x\in\{-1,1\}^n\}$ whereas $r^2_f\geq\max\{f(\x):\x\in\{-1,1\}^n\}$.
For more details, see e.g. 
\cite{lasserrebook1,lasserrebook2}. Next, if we define
\begin{equation}
\label{def-r-f}
\rho_f\,:=\,\max_{i=1,2}\vert r^i_{f}\vert,
\end{equation}
then it is straightforward to verify that $\rho_f\geq\max\{\,\vert f(\x)\vert:\x\in\{-1,1\}^n\,\}$.
Then we have the following analogue of Lemma \ref{lemma-1}:
 \begin{lemma}
\label{lemma-11}
Let $f^*$ be as (\ref{def-pb-f}) and
$\rho_f$ as in (\ref{def-r-f}). If $f^*<+\infty$ then 
\begin{equation}
\label{poly-pb}
f^*=\rho\,:=\,\min_{\x\in \{-1,1\}^n}\:f(\x)+(2\,\rho_f+1)\cdot\Vert \A\,\x-\b\Vert^2
\end{equation}
(a polynomial optimization problem on$\{-1,1\}^n$).
Moreover $f^*=+\infty$ if and only if $\rho>\rho_f$.
\end{lemma}
\noindent
The proof being almost a verbatim copy of that of Lemma \ref{lemma-1}, is omitted.

 As for the quadratic case and with same arguments, one may also show that if $d$ is even, the polynomial optimization problem
 (\ref{poly-pb}) is equivalent to minimizing the homogeneous polynomial $\tilde{f}$ of degree $d$ on the hypercube
 $\{-1,1\}^{n+1}$, where
 \[(\x,x_0)\mapsto \tilde{f}(\x):=x_0^{d}f(\x/x_0)+(2\rho_f+1)\,x_0^{d-2}\cdot\Vert \A\,\x-x_0\,\b\Vert^2.\]
 But since it is not a MAX-CUT problem, to obtain a lower bound on $f^*$ one may just as well consider solving the first level of the Lasserre-SOS hierarchy associated with (\ref{poly-pb}) or even directly with (\ref{def-pb-f}). The advantage of using 
 the formulation (\ref{poly-pb}) is that one always minimizes on the hypercube $\{-1,1\}^n$ instead of minimizing on 
 the subset  $\{-1,1\}^{n}\cap \{\x:\A\,\x=\b\}$ of the hypercube which is problem dependent.

\section{Some numerical experiments}

All semidefinite programs were solved by using the GloptiPoly software
dedicated to solving the Generalized Problem of Moments and which implements the 
Lasserre-SOS hierarchy of semidefinite relaxations. For more details the interested reader is referred to 
Henrion et al. \cite{gloptipoly}. Typically, for a MAX-CUT problem of size $n=100$,
solving the Shor relaxation with GloptiPoly (and calling the semidefinite solver MOSEK \cite{mosek})
takes approximately 42s on a Macbook-pro lap-top with Intel Core i5 processor at 2.6 GHz.

\subsection{Some $0/1$-knapsack examples}
\label{0-1-knapsack}
To test the efficiency of the MAX-CUT formulation we have first
considered $0/1$-knapsack problems of small size (up to $n=15$ variables) and
compared the first semidefinite relaxation (or Shor-relaxation) of the MAX-CUT formulation
(\ref{quad-pb}) with some other relaxations, and in particular with:

- The LP-relaxation of (\ref{def-newpb}) (when $\F=0$) which consists of replacing  the constraint
$\x\in\{-1,1\}^n$ with $\x\in [-1,1]^n$, and

- The first semidefinite relaxation of (\ref{def-newpb}) in the Lasserre-SOS hierarchy.
\begin{example}
\label{ex1}
To evaluate the quality of the lower bound obtained with the MAX-CUT formulation consider 
the following simple linear knapsack-type examples:
\begin{equation}
\label{knapsack}
\min\,\{ \c^T\x: \:\a^T\x=b;\:\x\in \{-1,1\}^n\,\},\end{equation}
on $\{-1,1\}^n$, with $4$ and $10$ variables. 
For $n=4$, $\c=(13,\,11,\, 7,\, 3)$ and $\a=(3,\, 7,\, 11,\, 13)$,
while for $n=10$, $\c=(37,\, 31,\, 29,\, 23,\, 19 ,\,17,\, 13,\, 11,\, 7,\, 3)$,
and $\a=(3,\, 7 ,\,11,\, 13,\, 17,\, 19,\, 23,\, 29,\, 31,\, 37)$ and for $n=15$
\[\a=(3,\, 7 ,\,11,\, 13,\, 17,\, 19,\, 23,\, 29,\, 31,\, 37,\,41,\,43,\,47,\,51,\,53)\] and 
$\c=(53,\,51,\,47,\,43,\,41,\,37,\, 31,\, 29,\, 23,\, 19 ,\,17,\, 13,\, 11,\, 7,\, 3)$.

The right-hand-side $b$ is taken into $[\,-\vert \a\vert,\vert\a\vert\,]\cap \Z$. 
Figure \ref{fig1} displays the difference $\min\Q_+-\min{\rm LP}$ 
where the lower bound $\min{\rm LP}$ is obtained by relaxing the 
integrality constraints $\x\in\{-1,1\}^n$ to the box constraint $\x\in [-1,1]^n$ and solving the resulting LP.
On the $x$-axis one reads $b+\vert\a\vert$. As expected the lower bound $\min\Q_+$ is much better than $\min{\rm LP}$.
In fact the cases where 
the LP-bound is slightly better is for right-hand-side $b$ such that the relaxation provides the optimal value $f^*$.
\begin{figure}[htbp]
\begin{center}
\resizebox{1\textwidth}{!}
{\includegraphics{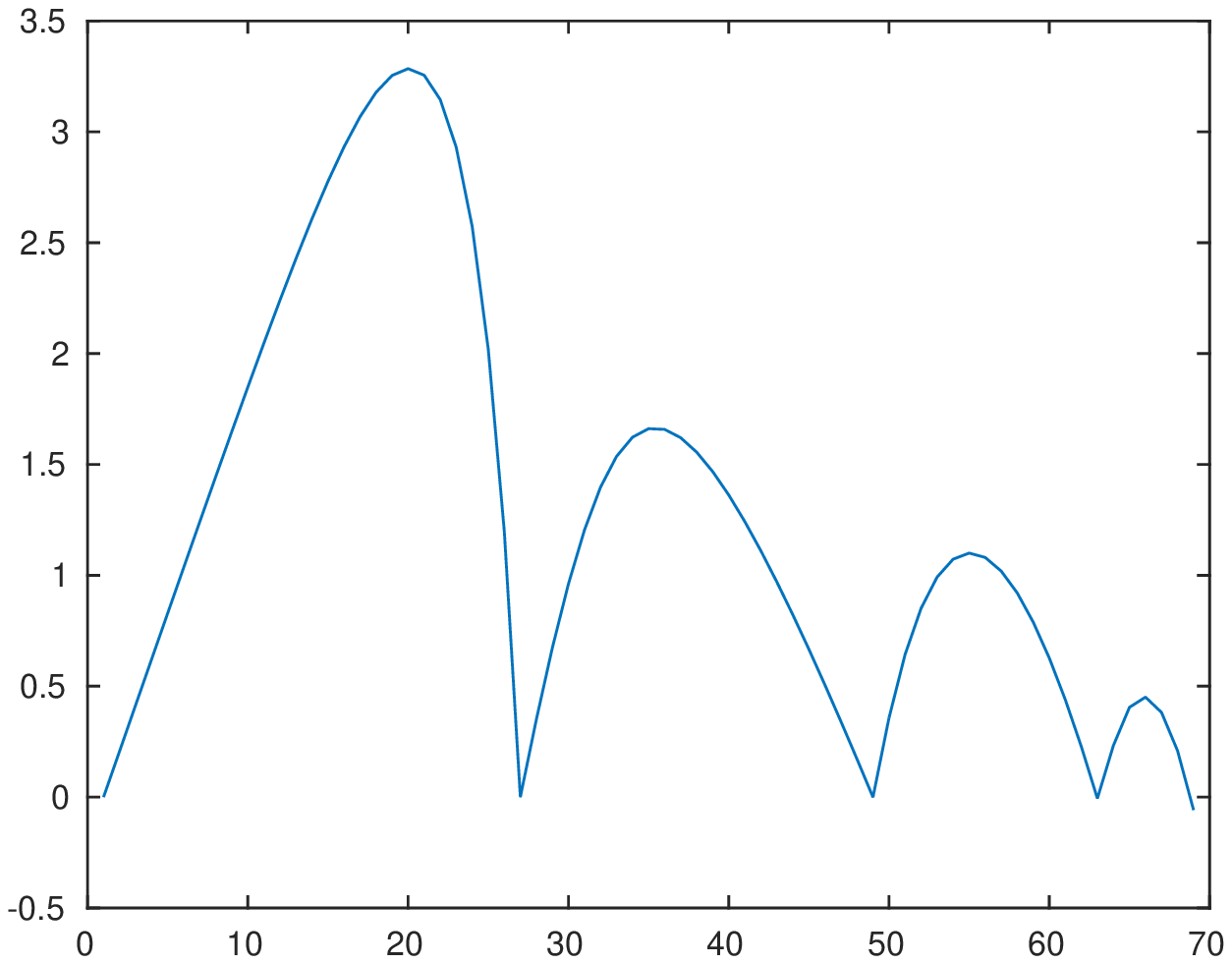}\includegraphics{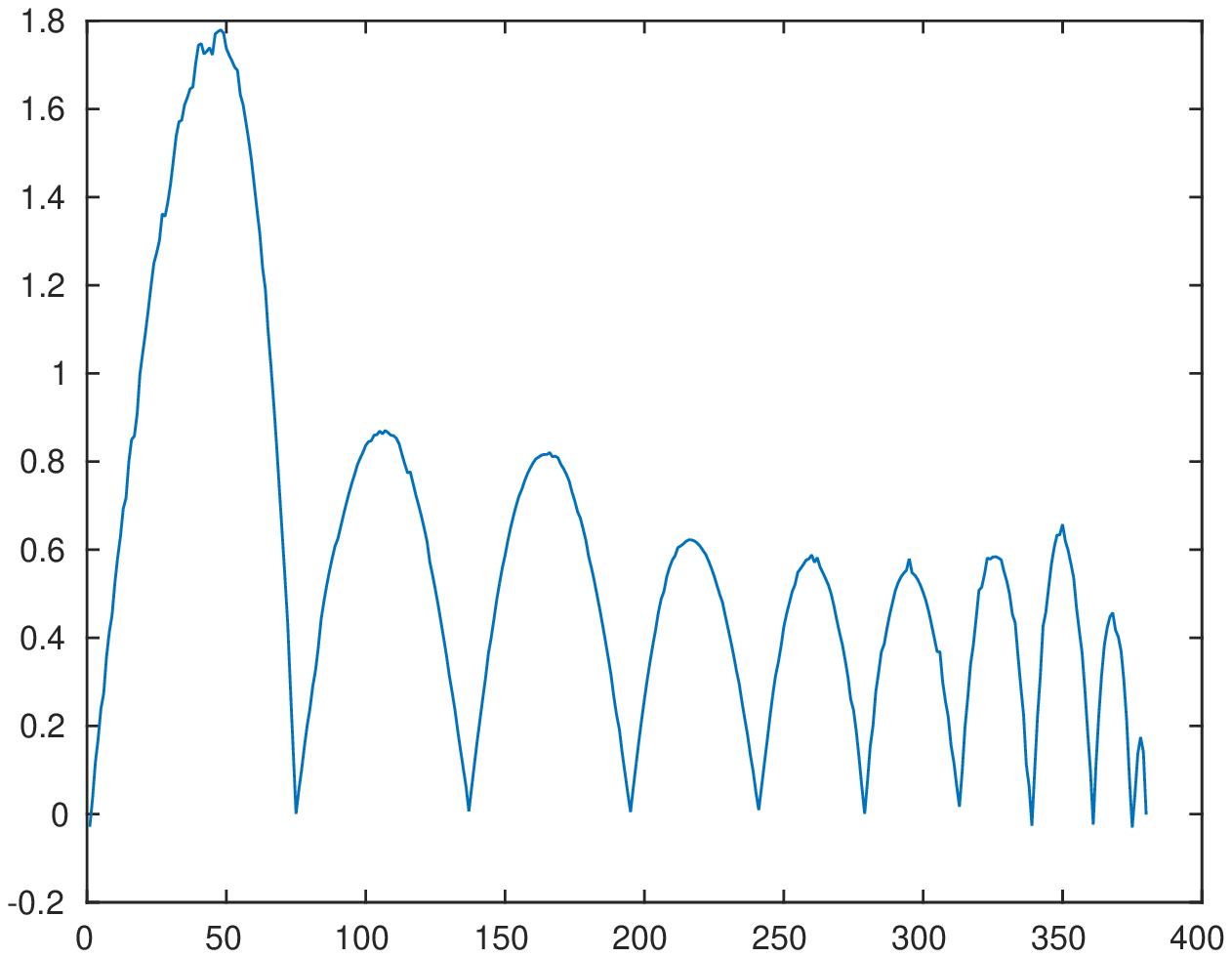}
\includegraphics{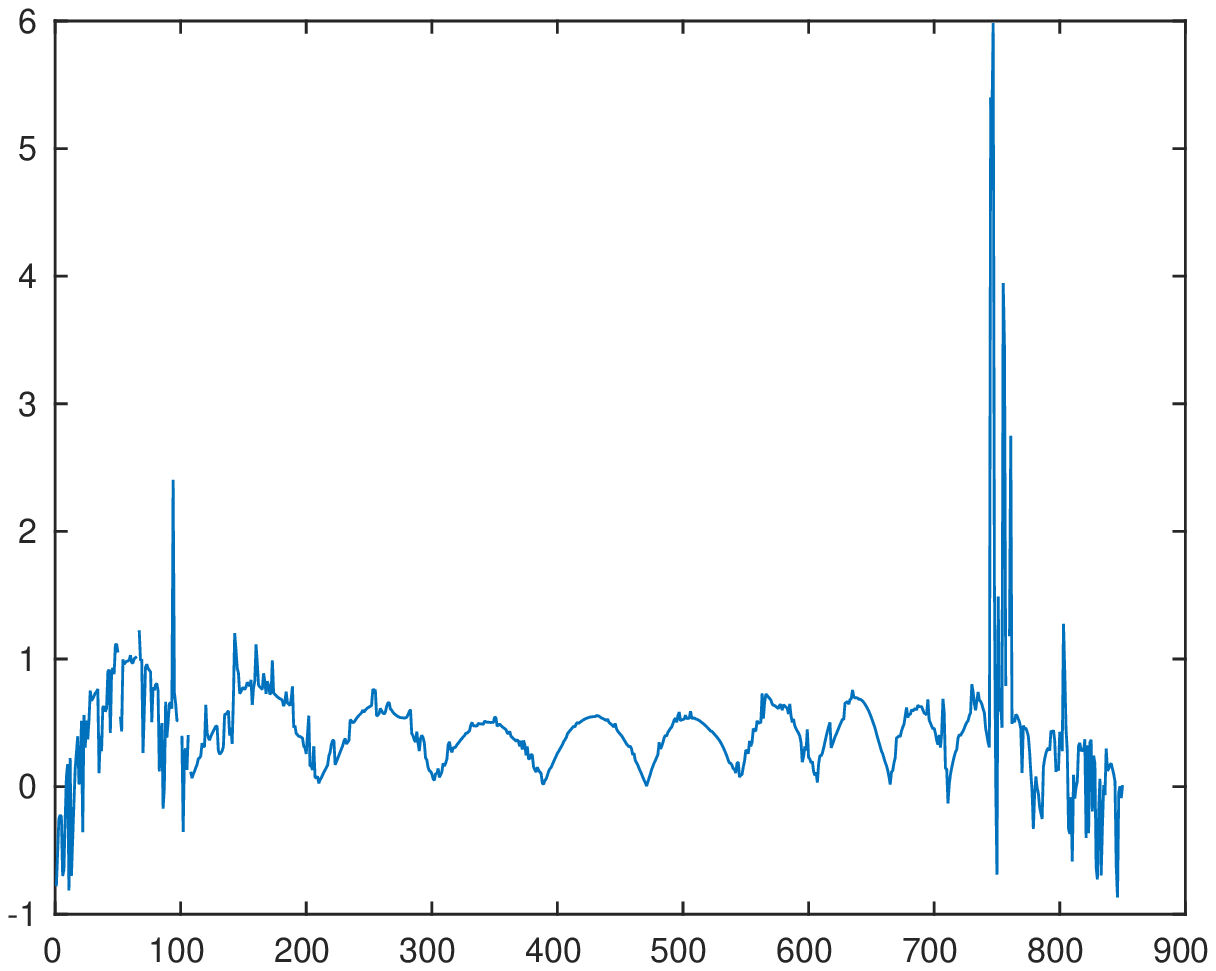}}
\caption{Difference $\min\Q_+-\min{\rm LP}$ with n=(4,10,15); One reads $b+\vert\a\vert$
on the $x$-axis}
\label{fig1}
\end{center}
\end{figure}

Moreover Figure \ref{fig2} displays the difference $\min\Q_+-\min\hat{\Q}_+$ where
$\min\hat{\Q}_+$ is the optimal value of the first SDP-relaxation of the Lasserre-SOS hierarchy
applied to the initial formulation (\ref{knapsack}) of the knapsack problem
where one has even included the redundant constraints $x_i\,(\a^T\x-b)=0$, $i=1,\ldots,n$.
One observes that in most cases the lower bound $\min\Q_+$ is slightly better than $\min\hat{\Q}_+$.

This is encouraging since the Lasserre-SOS hierarchy \cite{lasserre1,lasserre2} is known to produce
good lower bounds in general, and especially at the first level of the hierarchy 
for MAX-CUT problems whose matrix $\Q$ of the associated 
quadratic form has certain properties, e.g.,  $Q_{ij}\geq0$ for all $i,j$ or $\Q\succeq0$ (in the maximizing case);
see e.g. Marshall \cite{marshall}.
\begin{figure}[htbp]
\begin{center}
\resizebox{0.8\textwidth}{!}
{\includegraphics{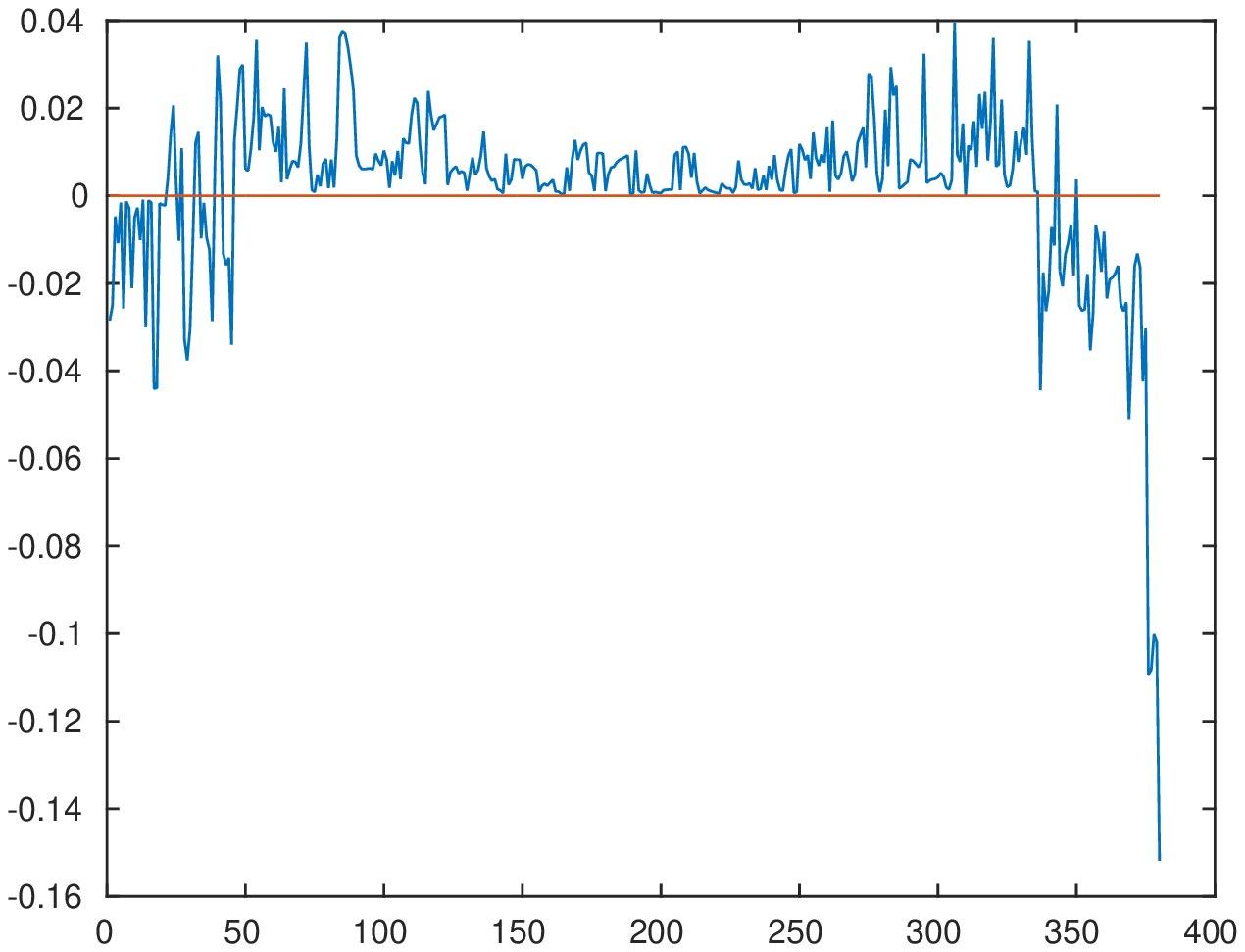}
\includegraphics{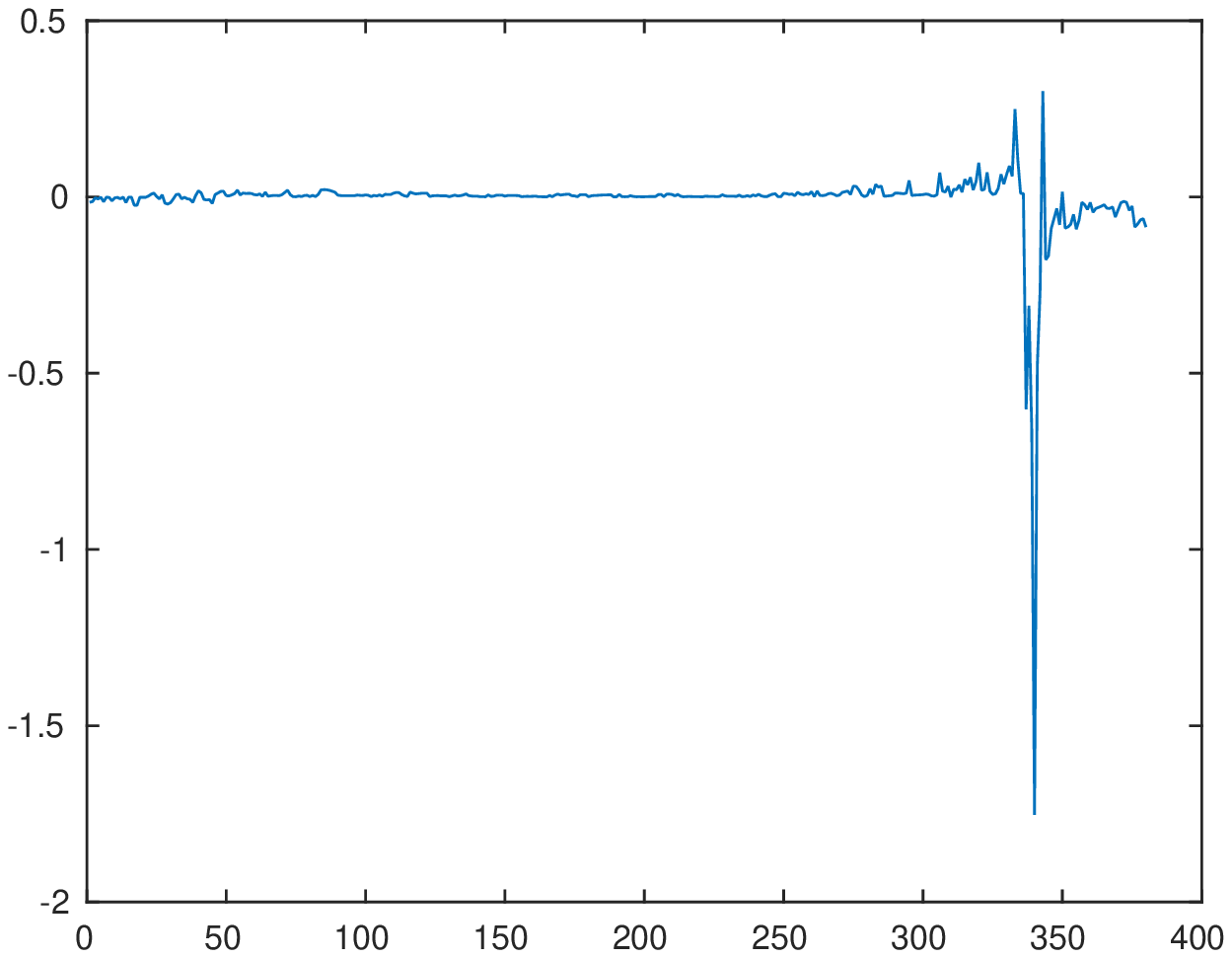}}
\caption{Difference $\min\Q_+-\min\hat{\Q}_+$ (n=10) (left) and relative difference
$100*(\min\Q_+-\min\hat{\Q}_+)/\min\hat{\Q}_+$; One reads $b+\vert\a\vert$ on the $x$-axis}
\label{fig2}
\end{center}
\end{figure}
\end{example}
\begin{example}
\label{ex2}
In a second sample of linear knapsack problems (\ref{knapsack}) with $n=10,15$, we have chosen the same vector $\a$
as in Example \ref{ex1} but now with a cost criterion of the form $c(i)=a(i)+s\,\eta$, $i=1,\ldots,n$, where $\eta$ is a random variable
uniformly distributed on $[0,1]$, and $s$ is a weighting factor. The reason is that knapsack problems with ratii $c(i)/a(i)\approx 1$ for all $i$, can be difficult to solve. As before the right-hand-side $b$ is taken into $[\,-\vert \a\vert,\vert\a\vert\,]\cap \Z$. 
Figure  \ref{fig4} displays results obtained for $s=10$, and $n=10,15$.
\begin{figure}[htbp]
\begin{center}
\resizebox{0.8\textwidth}{!}
{\includegraphics{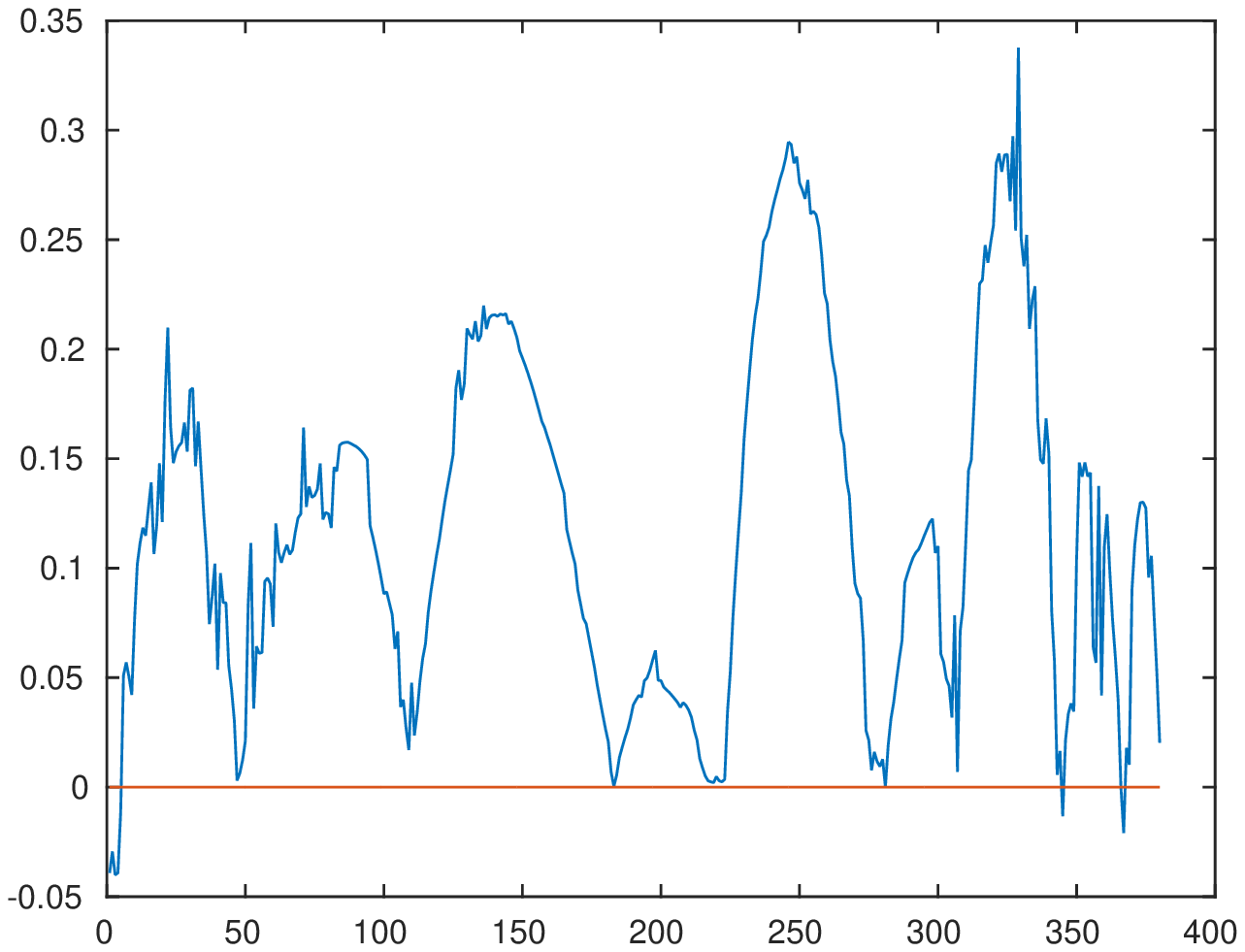}
\includegraphics{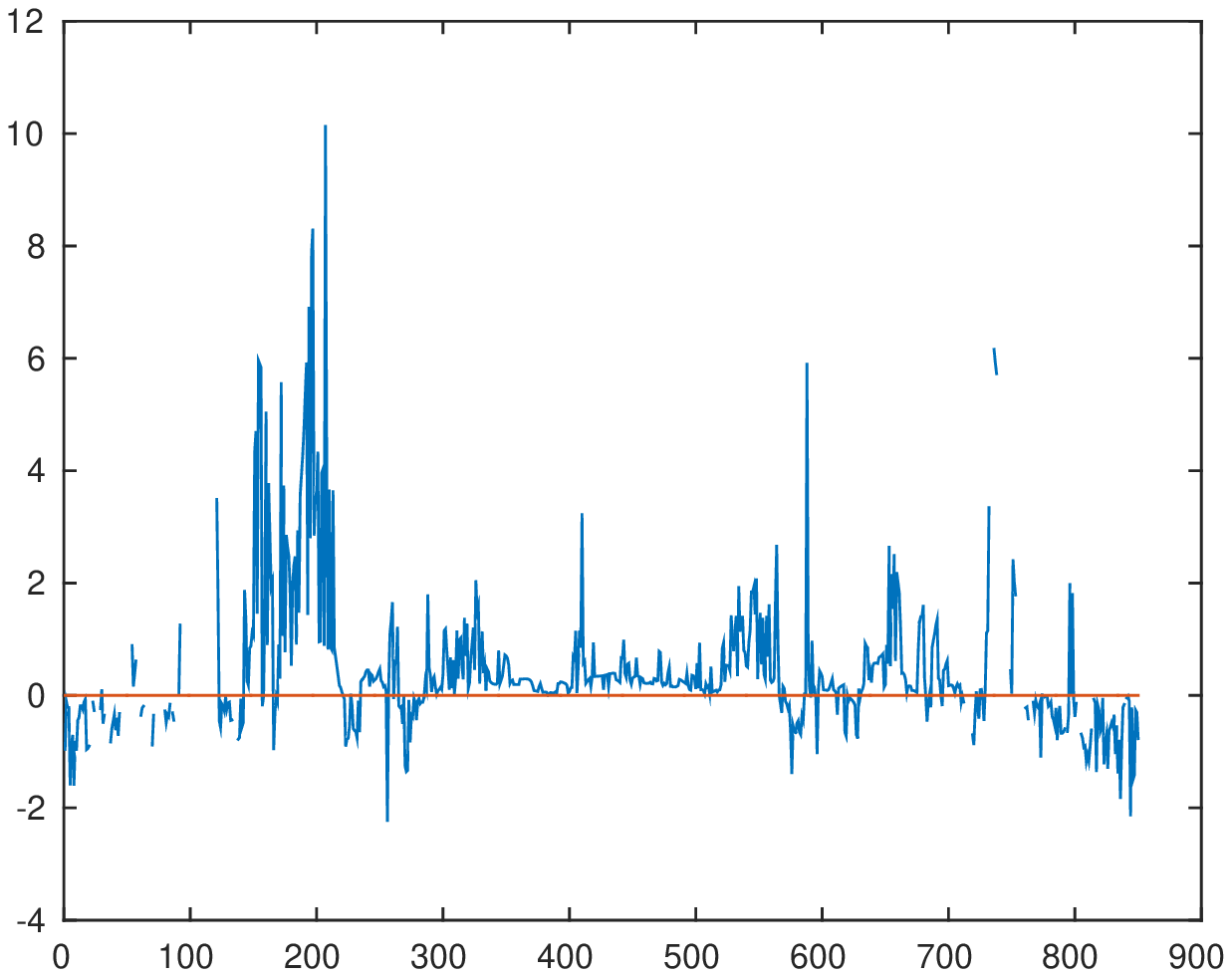}}
\caption{Difference $\min\Q_+-\min{\rm LP}$, $\c=\a+10*\eta$ (n=10,15); $b+\vert\a\vert$ on the $x$-axis}
\label{fig4}
\end{center}
\end{figure}
\end{example}
Finally, as for Example \ref{ex1}, Figure \ref{fig6} displays the difference $\min\Q_+-\min\hat{\Q}_+$ (where
$\min\hat{\Q}_+$ is the optimal value of the first SDP-relaxation of the Lasserre-SOS hierarchy
applied to the initial formulation (\ref{knapsack}) of the knapsack problem
where one has even included the redundant constraints $x_i\,(\a^T\x-b)=0$, $i=1,\ldots,n$).
Again one observes that in most cases the lower bound $\min\Q_+$ is slightly better than $\min\hat{\Q}_+$.
\begin{figure}[htbp]
\begin{center}
\resizebox{0.8\textwidth}{!}
{\includegraphics{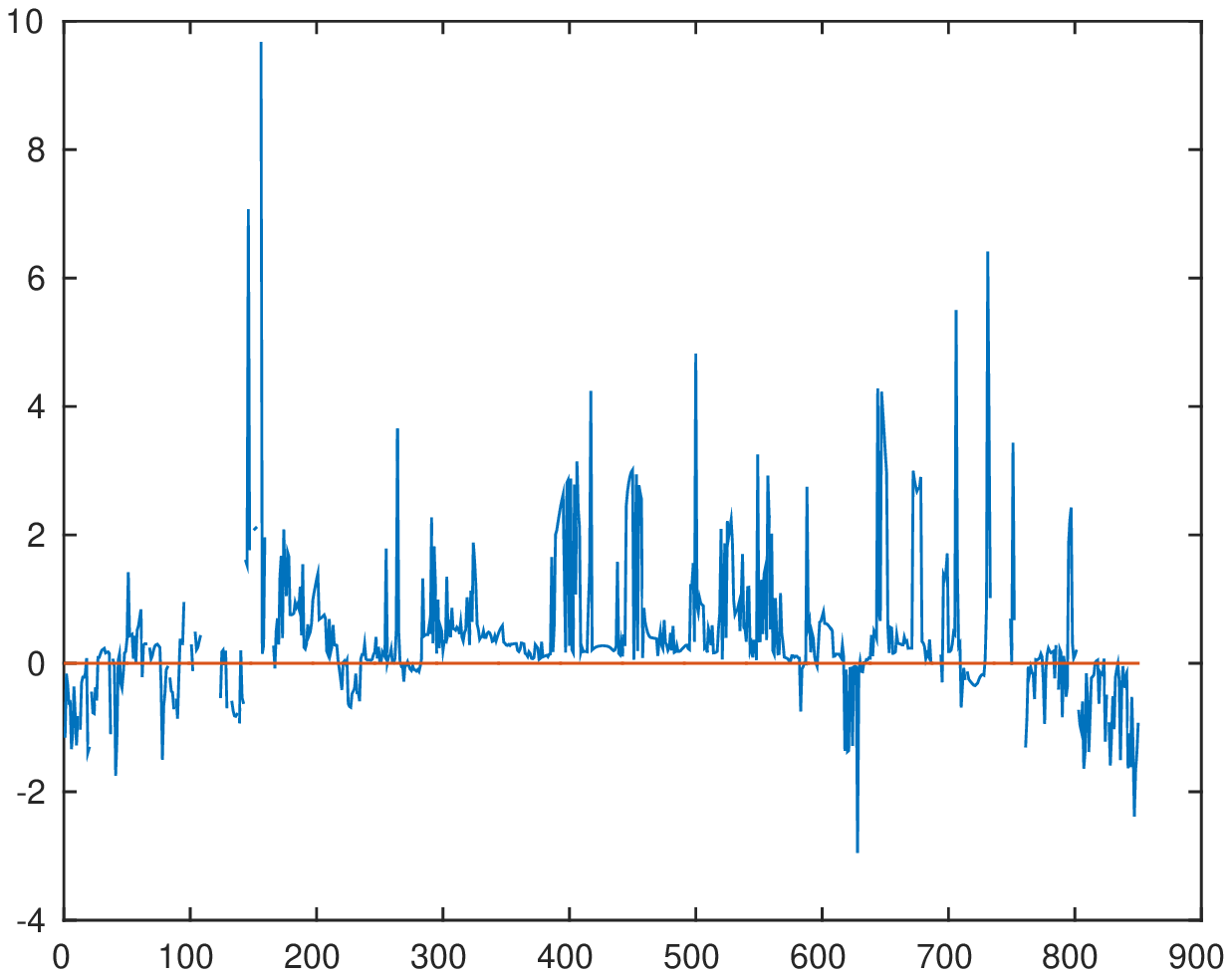}\includegraphics{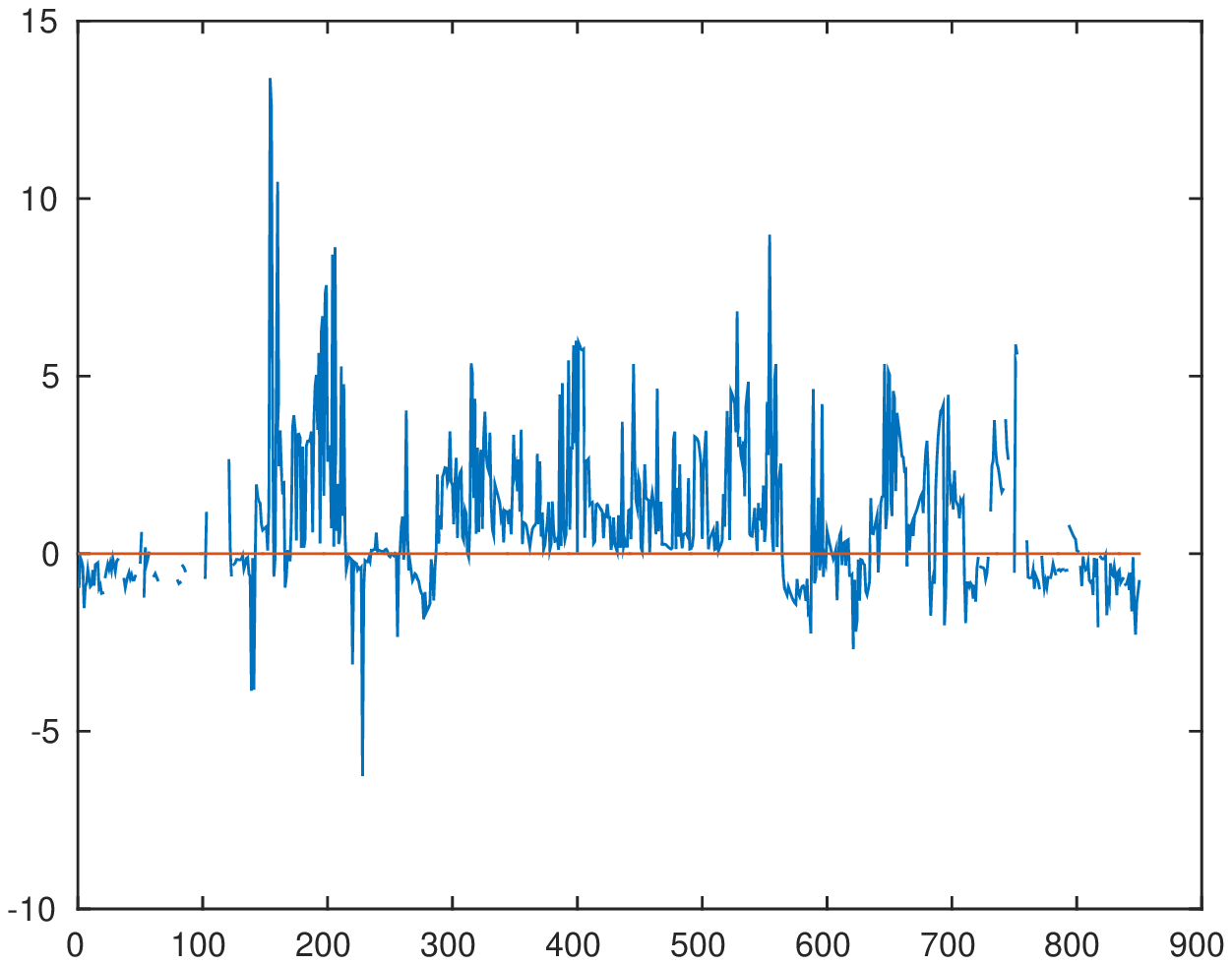}}
\caption{Example \ref{ex2}: Difference $\min\Q_+-\min\hat{\Q}_+$, (n=15) $\c=\a+20\eta$ 
(left) and $\c=\a+10\eta$; $b+\vert\a\vert$ on the $x$-axis }
\label{fig6}
\end{center}
\end{figure}
\begin{example}
\label{ex3}
We next consider the same knapsack problems (\ref{knapsack}) as in Example \ref{ex2} but now with quadratic criterion
$\c^T\x+\x^T\F\x$, again with a cost criterion of the form $c(i)=a(i)+s\,\eta$, $i=1,\ldots,n$, where $\eta$ is a random variable
uniformly distributed in $[0,1]$. The real symmetric matrix $\F$ is also randomly generated
and is not positive definite in general. Again in Figure \ref{fig8} one observes that the lower bound $\min\Q_+$
is almost always better than the optimal value $\min\hat{\Q}_+$ of the first level of the Lasserre-SOS hierarchy applied to the original formulation (\ref{knapsack}) of the problem.

\begin{figure}[htbp]
\begin{center}
\resizebox{0.5\textwidth}{!}
{\includegraphics{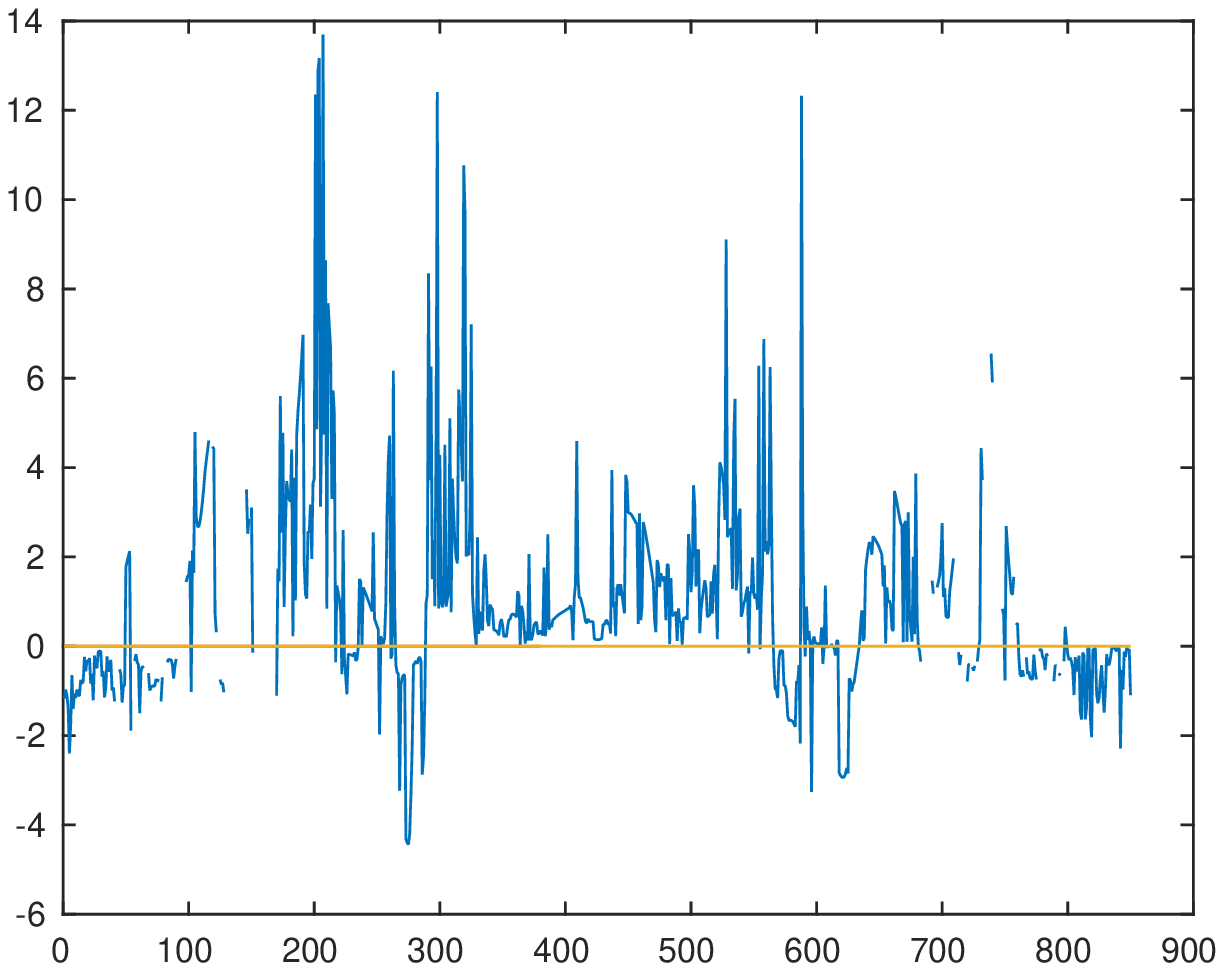}}
\caption{Example \ref{ex3}: Difference $\min\Q_+-\min\hat{\Q}_+$, (n=15) 
$\c=\a+10\eta$; $b+\vert\a\vert$ on the $x$-axis }
\label{fig8}
\end{center}
\end{figure}
\end{example}

\subsection{Larger size problems: The $k$-cluster problem}
As generating data $(\c,\a,b)$ for difficult $0/1$-knapsack problems of larger size
is not obvious, we next consider the so-called {\it $k$-cluster} problem in $n=70$ and $n=80$ variables:
\begin{equation}
\label{original}
 \min\,\{\,\x^T\A\x:\:\e^T\,\x\,=\,k;\:\x\in\{0,1\}^n\,\},\end{equation}
where $k\in\N$ is a fixed parameter, $\e=(1,1,\ldots,1)\in\R^n$, $\A\in\R^{n\times n}$ 
with $\A=\A^T$ and $\A_{ii}=0$ for all $i=1,\ldots,n$. Indeed 
the constraint $\e^T\x=k$  of the  $k$-cluster problem
is structural and does not depend on a vector $\a$ as in knapsack examples.
After the change of variables
$\x\to (\x+\e)/2$ the $k$-cluster problem reads:
\begin{equation}
\label{k-cluster-example}
\P:\quad\min\,\{\,2\e^T\A\x+\x^T\A\x:\:\e^T\x\,=\,2k-n;\:\x\in\{-1,1\}^n\,\}.
\end{equation}
We have compared the optimal value $f^*_{maxcut}$ for the Shor relaxation of the MAX-CUT formulation
(\ref{quad-pb}) of (\ref{k-cluster-example}) with the optimal value $f^*_{Shor}$ for  
the first semidefinite relaxation of 
(\ref{k-cluster-example}) in the Lasserre-SOS hierarchy. The latter
is a rather basic quadratic relaxation of the initial problem. In all our experiments
with $n=70$ and $n=80$, 
both values are almost identical since the relative difference is 
no more than $0.014\%$ and $0.06\%$ respectively.

Of course when $k>n$ the $k$-cluster problem (\ref{k-cluster-example})
has no feasible solution to $\P$. In view of the special structure of the $k$-cluster problem
it is straightforward to check that the first semidefinite relaxation of 
(\ref{k-cluster-example}) (or equivalently of (\ref{original})) in the Lasserre-SOS hierarchy is infeasible, hence certifying 
that $f^*=+\infty$. On the other hand the Shor relaxation 
of the MAX-CUT formulation of (\ref{k-cluster-example})
has always a finite value. However, one may prove that if $k>n$ then the optimal value is larger than $\rho(\c,\F)$, which by Proposition \ref{prop-main} also certifies that $f^*=+\infty$. That is, for the $k$-cluster problem (\ref{k-cluster-example}) the Shor relaxation of its MAX-CUT formulation (\ref{quad-pb})
also detects infeasibility.

\subsection*{Comparing with a quadratic convex relaxation}

Decompose $\A$ as $\U^T\Lambda\U$ with $\U^T\U=\mathbf{I}$ (the identity matrix) and $\Lambda$ is the diagonal matrix of 
eigenvalues of
 $\A$. Let $\theta:=\vert\lambda_{min}(\A)\vert+1$ where $\lambda_{min}(\A)$ is the smallest eigenvalue of $\A$.
 Replace 
$\A$ with $\tilde{\A}:=\U^T(\Lambda+\theta\mathbf{I})\U$
so that the resulting the quadratic form $\x\mapsto \x^T\tilde{\A}\x$ is convex. Then observe that
\[\min\,\{\,2\e^T\A\x+\x^T\tilde{\A}\x:\:\e^T\x\,=\,2k-n;\:\x\in\{-1,1\}^n\,\}\,=\,f^*+n\theta\]
because $\x^T\tilde{\A}\x=\x^T\A\x+\theta\Vert\x\Vert^2$ and $x_i^2=1$ for all $i=1,\ldots,n$.
Therefore one may compare the optimal value
$f^*_{maxcut}$ of the Shor relaxation of (\ref{k-cluster-example}) with $\tilde{\psi}=\psi-n\theta$, where 
\[\psi=\min\,\{\,2\e^T\A\x+\x^T\tilde{\A}\x:\:\e^T\x\,=\,2k-n;\: x_i^2\leq 1,\:i=1,\ldots,n\,\},\]
which is a convex quadratic relaxation. Results displayed in Table \ref{table-kcluster70-qconvex} 
show that the Shor relaxation is always better, and sometimes significantly better. However this convex relaxation 
is rather weak.

Next we do the same comparison for a $k$-cluster problem
where some sparsity is introduced in the cost matrix $\A$.
Namely once $\A$ has been generated then for each
non-diagonal entry $\A_{ij}$ we decide
$\A_{ij}=\A_{ji}:=0$ with probability $0.4$, $0.7$ and $0.8$.
Results displayed in Table \ref{table-kcluster100-qconvex-rand}  for $n=100$ 
show that again the Shor relaxation is always better 
(and some times significantly better) than the quadratic convex relaxation.
The sparser is $\A$ the more efficient is the Shor relaxation compared with the quadratic convex relaxation.

\begin{table}[!h]
\begin{center}
\begin{tabular}{||c|c|c|c|c|c|c||}
\hline
n=70	 & k=50 &k=45& k=40 &k=35 &k=30 & k=25\\
\hline
$f^*_{maxcut}$ &   2324.68&   1350.74&   488.83&  -261.78&  -900.55  & -1429.66\\
$\tilde{\psi}$ &   2260.00&   1278.08&   410.48&  -343.95&  -985.57&  -1514.77\\
$\frac{100(f^*_{maxcut}-\tilde{\psi})}{\tilde{\psi}}$&   2.86\%&   5.68\%&  19.08\%&   23.89\%&   8.62\% &5.61\%\\
\hline
\end{tabular}
\end{center}
\caption{$k$-cluster: Shor relaxation ($f^*_{maxcut}$) vs convex quadratic relaxation ($\tilde{\psi}=\psi-n\theta$) with $n=70$ 
\label{table-kcluster70-qconvex}}
\end{table}

\begin{table}[!h]
\begin{center}
\begin{tabular}{||c|c|c|c|c|c||}
\hline
n=100&	k=70 & k=60 &k=50& k=40& k=20 \\
\hline
&\multicolumn{3}{|c||}{$\A_{ij}=0$ with probability 0.8}\\
\hline
$f^*_{maxcut}$ & 513.32& -60.83&  -450.91&  -781.84 & -1084.98\\
$\tilde{\psi}$ & 419.07&  -147.77& -572.27&  -894.72&  -1188.59\\
$\frac{100(f^*_{maxcut}-\tilde{\psi})}{\tilde{\psi}}$
& 22.4\%&  58.8\%& 21.2\%& 12.6\%&8.7\%\\
\hline
\hline
&\multicolumn{3}{|c||}{$\A_{ij}=0$ with probability 0.6}\\
\hline
$f^*_{maxcut}$ &  1278.01& 256.43 & -627.72& -1248.39& -1988.94\\
 $\tilde{\psi}$ &  1192.99& 155.44&  -723.34&  -1372.94& -2095.66\\
$\frac{100(f^*_{maxcut}-\tilde{\psi})}{\tilde{\psi}}$ & 7.1\%&  64.9\%& 13.2\%&   9.0\%&   5.0\%\\
\hline
\hline
 \end{tabular}
\end{center}
\caption{$k$-cluster: Shor relaxation ($f^*_{maxcut}$) vs convex quadratic relaxation ($\tilde{\psi}=\psi-n\theta$) with $n=100$. 
$\A_{ij}=0$ with probability $p=0.8,0.6$.
\label{table-kcluster100-qconvex-rand}}
\end{table}

\newpage

\subsection{Comparing with the copositive formulation}
 
 As already  mentioned in the introduction, the 0/1 program (\ref{def-pb}) also has a copositive formulation. Namely, 
 let $e_i=(\delta_{i=j})\in\R^n$, $i=1,\ldots,n$, and  $\e=(1,\ldots,1)\in\R^{n}$. Following
 Burer \cite[p. 481--482]{burer}, introduce $n$ additional variables $\z=(z_1,\ldots,z_n)$ and the $n$ additional 
 equality constraints  $x_i+z_i=1$, $i=1,\ldots,n$, with $\z\geq0$ (which are necessary to obtain an equivalent formulation). So let $\tilde{\x}^T=(\x^T,\z^T)\in\R^{2n}$ and with $\Id\in\R^{n\times n}$ 
 being the identity matrix,
 introduce the real matrices
 \[\tilde{\F}:=\left[\begin{array}{cc}\F &0\\
 0&0\end{array}\right],\quad \S:=\left[\begin{array}{cc}\A &0\\
 \Id&\Id\end{array}\right]\]
 and the real vectors $\tilde{\c}^T:=(\c^T,0)\in\R^{2n}$ and $\tilde{\b}^T=(\b^T,\e^T)\in\R^{m+n}$. Let  $\S_i$ denote the $i$-th row vector of $\S$, $i=1,\ldots,2n$.
 Then the copositive formulation of (\ref{def-pb}) reads:
 \begin{equation}
 \label{copo}
 \begin{array}{rl}f^*=\min &\tilde{\c}^T\tilde{\x}+\langle \tilde{\F},\X\rangle\\
 \mbox{s.t.}&\S_i\,\tilde{\x}\,=\,\tilde{\b}_i;\quad
 \S_i\,\X\,\S_i^T=\tilde{\b}_i^2,\quad i=1,\ldots,m+n\\
 &\X_{ii}=\tilde{\x}_i,\quad i=1,\ldots,2n;\:\left[\begin{array}{cc}1&\tilde{\x}^T\\ \tilde{\x}&\X\end{array}\right]\,\in\,\mathcal{C}^*_{2n+1},
  \end{array}\end{equation}
 where $\mathcal{C}_{2n+1}$ is the convex cone of $(2n+1)\times (2n+1)$
 copositive matrices and $\mathcal{C}^*_{2n+1}$
 is its dual, i.e., the convex cone of {\it completely positive matrices}.

The hard constraint being membership in $\mathcal{C}^*_{2n+1}$, a strategy is to use
hierarchies of tractable approximations (of increasing size) of $\mathcal{C}^*_{2n+1}$, as described in e.g. D\"ur \cite{durr}.
In particular a possible choice for the first relaxation in such hierarchies is to replace (\ref{copo}) with the semidefinite program:
\begin{equation}
 \label{copo-relax}
 \begin{array}{rl}f^*_{copo}=\min &\tilde{\c}^T\tilde{\x}+\langle \tilde{\F},\X\rangle\\
 \mbox{s.t.}&\S_i\,\tilde{\x}\,=\,\tilde{\b}_i;\quad  
 \S_i\,\X\,\S_i^T=\tilde{\b}_i^2,\quad i=1,\ldots,m+n\\
 &\X_{ii}=\tilde{\x}_i,\quad i=1,\ldots,2n;\:
 \left[\begin{array}{cc}1&\tilde{\x}^T\\ \tilde{\x}&\X\end{array}\right]\,\in\,\mathcal{S}^+_{2n+1}\cap\mathcal{N}_{2n+1},
 \end{array}\end{equation}
where $\mathcal{S}^+_{2n+1}$ (resp. $\mathcal{N}_{2n+1}$) is the convex cone
of real symmetric positive semidefinite (resp. entrywise nonnegative) matrices. Then
(\ref{copo-relax}) is a semidefinite relaxation of (\ref{copo}) because
$\mathcal{C}^*_{2n+1}\subset\mathcal{S}^+_{2n+1}\cap\mathcal{N}_{2n+1}$, and so $f^*_{copo}\leq f^*$.
\begin{example}
\label{ex-copo-final}  
 We have considered the $k$-cluster problem
 (\ref{original}) with $n=50$ variables and where $\F$ is randomly generated with $80\%$ of zero entries in $\F$.
Then we compared the optimal values $f^*_{maxcut}$ and $f^*_{copo}$ of the respective Shor-relaxations of the MAX-CUT formulation (\ref{quad-pb}) and
the copositive formulation of (\ref{original}).
 For most values of $k$ the lower bound $f^*_{maxcut}$ was almost identical to $f^*_{copo}$
  (the relative difference $\Delta:=100\,(f^*_{maxcut}-f^*_{copo})/f^*_{copo}$ being at most $0.1\%$), and was better 
 ($\Delta=68\%$ and $44\%$) for small $k=12$ and $k=14$ respectively. However notice that the size of the semidefinite constraint 
in the latter is twice as large as the one in the former.
\end{example}
\begin{remark}
Notice that in all examples the bounds $f^*_{maxcut}$ and $f^*_{Shor}$  
(respectively denoted $\min\Q_+$ and $\min\hat{\Q}_+$ in Section \ref{0-1-knapsack}) are almost identical. It might be tempting to conjecture that they
are indeed identical and the error is coming from numerical roundoffs. However results displayed in Figure \ref{fig8}
show that the difference $\min\Q_+-\min\hat{\Q}_+$  is not uniformly distributed.
\end{remark}
 \section{Conclusion}
 
 In this paper we have shown that a linear or quadratic 0/1 program $\P$ has an equivalent MAX-CUT 
 formulation and so the whole arsenal of approximation techniques for the latter can be applied.
 In particular, and as suggested by some preliminary tests on a  (limited sample) of 0/1 knapsack and 
 $k$-cluster examples,
 it is expected that the lower bound obtained from the Shor relaxation of this MAX-CUT formulation
 will be in general better than the one obtained from the standard LP-relaxation (for linear 0/1 programs)
 of the original problem. The situation might be even better for quadratic 0/1 programs when
the Shor relaxation of the MAX-CUT formulation is compared with a convex quadratic relaxation.
 
 \subsection*{Acknowledgements}
 {\small This research was supported by a grant from the MONGE program of the
{\it F\'ederation Math\'ematique Jacques Hadamard} (FMJH, Paris) and an ERC Advanced Grant
(ERC-ADG TAMING 666981) of the European Research Council (ERC). The author also thanks anonymous referees for helpful suggestions 
to improve the initial version of the paper.}

\end{document}